\documentclass[preprint,12pt]{elsarticle}

\usepackage{amssymb, amsmath, amsthm, amsfonts}
\usepackage{graphicx, geometry, float, booktabs}
\usepackage{indentfirst}
\usepackage{mathrsfs, colortbl, color}
\usepackage{enumerate, enumitem}
\setenumerate{label=(\roman*), itemsep=0pt,partopsep=0pt,parsep=\parskip,topsep=0pt}

\newtheorem*{thm*}{Theorem}
\newtheorem{thm}{Theorem}[section]
\newtheorem{lem}[thm]{Lemma}
\newtheorem{prop}[thm]{Proposition}
\newtheorem{conj}[thm]{Conjecture}

\newtheorem{cor}[thm]{Corollary}

\newdefinition{rmk}{Remark}
\newproof{pf}{Proof}

\theoremstyle{definition}
\newtheorem{defi}[thm]{Definition}
\newtheorem{exm}[thm]{Example}
\AtEndEnvironment{exm}{\null\hfill\qedsymbol}

\newcommand{\BAA}{\mathbb A}
\newcommand{\FAA}{\mathfrak A}

\newcommand{\CBB}{\mathcal B}

\newcommand{\CNN}{\mathcal N}

\newcommand{\CQQ}{\mathcal Q}

\newcommand{\CSS}{\mathcal S}

\newcommand{\CN}{\mathbb{C}} 
\newcommand{\R}{\mathbb R}  
\newcommand{\Q}{\mathbb Q}
\newcommand{\Z}{\mathbb Z} 
\newcommand{\N}{\mathbb N} 

\newcommand{\red}{\textcolor{red}}

\journal{ -- }

\begin{document}

\begin{frontmatter}

\title{Entropic van der Corput's Difference Theorem}

\author[aff1]{Weichen Gu}
\author[aff1,aff2]{Xiang Li}
\affiliation[aff1]{organization={Department of Mathematics and Statistics, University of New Hampshire},
            city={Durham},
            postcode={03824},
            state={New Hampshire},
            country={USA}}
\affiliation[aff2]{organization={Integrated Applied Mathematics Program, University of New Hampshire},
            city={Durham},
            postcode={03824},
            state={New Hampshire},
            country={USA}}

\begin{abstract}
We prove an entropy version of van der Corput’s difference theorem: the entropy of a sequence is equal to the entropy of its differences.
This reveals a potential correspondence between the theory of uniform distribution mod 1 and entropy.
As applications, we establish the corresponding entropy versions for several other results on uniform distribution.
\end{abstract}


\end{frontmatter}

\section{Introduction}
This paper is concerned with the entropy of sequences.
Roughly speaking,
the entropy of a sequence is the topological entropy of its corresponding dynamical system,
describing a certain complexity of the sequence.
Another concept related to the complexity of sequences is uniform distribution.
In this paper, several results in the theory of uniform distribution are generalized to their entropy versions, which show that the two themes, entropy and uniform distribution, share a similar structure.

A main theorem in the theory of uniform distribution is van der Corput's Difference Theorem, which is traditionally stated as follows:

\begin{thm*}[Van der Corput's Difference Theorem,  \cite{van1931diophantische}]\label{thm:DT}
Let $\{x(n)\}_{n\in\N}$ be a sequence taking values in the torus $\R/\Z$.
Assume that for every $d \ge1$,
the sequence $\{x(n+d)-x(n)\}_{n\in\N}$ is uniformly distributed.
Then $\{x(n)\}_{n\in \N}$ is uniformly distributed.
\end{thm*}

Van der Corput's Difference Theorem contains a powerful idea of complexity reduction \cite{bergelson2016van}, and leads to many applications.
One of the applications is an inductive proof of Weyl’s equidistribution theorem \cite{weyl4},
stating that every polynomial sequence (with certain conditions) is uniformly distributed mod 1.
In the theory of entropy,
it is known that all polynomial sequences (mod 1) have zero entropy.
This allows us to obtain an entropy version of Weyl’s equidistribution theorem:
by replacing the conclusion of being ``uniformly distributed $\bmod$ $1$" in Weyl’s equidistribution theorem with having ``zero entropy", the statement still holds.
Similarly, another difference theorem of van der Corput in \cite{van1931diophantische} can also be generalized to an entropy version.
This difference theorem states that if the $k$-th difference of a monotone sequence converges to 0 (with certain conditions on the rate of the convergence),
then the sequence is uniformly distributed $\bmod$ $1$.
Substituting ``uniformly distributed $\bmod$ $1$" with ``zero entropy",
the entropy version statement was proved in \cite{GW}:
if the $k$-th difference of a sequence converges to 0,
then the sequence has zero entropy.

These comparisons lead to a situation in which the two topics, entropy and uniform distribution, share a very similar structure and appear to be ``dual'' to each other.
By this correspondence, van der Corput’s Difference Theorem asks the following question:
if all the differences of a given sequence have the same entropy, does the sequence itself have the same entropy as its differences?
Our first main result provides a positive answer to this question in a stronger form.

\begin{thm}[Entropic van der Corput's Theorem]\label{231211thm1}
Let $\{x(n)\}_{n\in\N}$ be a sequence taking values in the torus $\R/\Z$.
Then, for every $d\ge 1$, the sequence $\{x(n+d)-x(n)\}_{n\in\N}$ has the same entropy as $\{x(n)\}_{n\in\N}$.
\end{thm}

For the case when the sequence takes values in $\CN$,
we have the following corollary.

\begin{cor}\label{231210thm2}
Let $\{x(n)\}_{n\in\N}$ be a uniformly bounded sequence in $\mathbb{C}$.
Then the difference $\{x(n+1) - x(n)\}_{n\in\N}$ has zero entropy if and only if
$\{x(n)\}_{n\in\N}$ has zero entropy.
\end{cor}
Here, we remark that Corollary \ref{231210thm2} also positively answers a question in \cite{GW},
regarding whether a sequence is deterministic if its difference is deterministic.
Moreover, many results in \cite{GW} are applications of Corollary \ref{231210thm2}.

In the theory of uniform distribution,
another topic is the study of the distributions of sequences $\{a(n)x\}_{n\in \N}$,
where $\{a(n)\}_{n\in\N}$ is a given integer sequence and $x$ is a real number.
Therefore, it is also natural to consider their entropy.
As an application of Theorem \ref{231211thm1},
we have the following theorem.

\begin{thm}\label{240109thm1}
Let $\{a(n)\}_{n\in\N}$ be a sequence of integers such that $\{ a(n+1)-a(n)\}_{n\in \N}$ is bounded uniformly.
Then for almost all real numbers $x$, the sequences $a_x(n):= a(n)x \bmod 1 $ have the same entropy.
\end{thm}

When the condition on the differences of $a(n)$ is not satisfied,
we have the following result.


\begin{thm}\label{240112thm1}
Let $\{a(n)\}_{n\in\N}$ be an increasing sequence of distinct integers such that
$$\lim_n a(n+1)/a(n)= p,$$
where $p$ can be non-integers.
Then for almost all real numbers $x$, the sequence $a_x(n):= a(n)x \bmod 1 $ has entropy at least $\log p$.
\end{thm}

All these results lead us to define a notion of dual entropy (cf. Definition \ref{240117defi1}),
inspired by measuring the complexity of an integer sequence by its action on $\R/\Z$.
Now, for an increasing sequence $a(0)<a(1)<\cdots$ in $\N$, we have two different notions of entropy to measure its complexity:
one is the dual entropy of $a(n)$, if it exists;
the other one is the entropy of the characteristic function $1_A(n)$ of $A=\{a(0),a(1),...\}$.
We find that, in general, these two entropy are not equal.
One discussed example is the set of square-free numbers (see Proposition \ref{240116prop1}).
We also show that
these two entropy will not be significantly different (see Theorem \ref{240113thm1}.)

The paper is organized as follows.
Section \ref{sect:pre} reviews the definitions and
properties of (anqie) entropy that will be used throughout the paper.
In Section \ref{240116sec1} we prove Theorem \ref{231211thm1} and Corollary \ref{231210thm2}.
In Section \ref{240115sec1}, we discuss a specific case,
when $a(n) = p^n$ for some positive integer $p \geq 2$.
In Section \ref{section240103}, we prove Theorem \ref{240109thm1} and Theorem \ref{240112thm1}, and give an example of a sequence $a(n)$ which satisfies conditions in Theorem \ref{240112thm1} but $a_x(n)$ has entropy arbitrarily large for almost all $x$.
In Section \ref{section_twoentropy}, we provide some properties of the dual entropy and compare it with anqie entropy.
The dual entropy of the square-free sequence is calculated there.
In Section \ref{sect:furstenberg}, we discuss the relation between the entropy and Furstenberg’s
$\times2\times3$ conjecture.
Finally, we offer some discussions on the M\"{o}bius disjointness conjecture in Section \ref{sect:sarnak}.

\section{Preliminaries on anqie entropy of sequences} \label{sect:pre}

In this section, we shall revisit the definition of anqie and anqie entropy, and discuss some properties of anqie entropy.
We refer readers to \cite{ge2016,wei2022anqie} for more details.

There are multiple ways to construct a dynamical system for a given sequence,
all of which yield the same dynamical system and, consequently, the same topological entropy.
In this paper, we construct the dynamical system mainly from the perspective of C*-algebras.
This point of view was first introduced by Ge in \cite{ge2016}.
The reason we prefer to use C*-algebras is that the difference operator is related to the algebraic structure.
The dynamical system of the sequence (as well as its corresponding C*-subalgebra) is called anqie,
and the entropy of this dynamical system is called anqie entropy.
Note that when the sequence takes values from a finite set, such as the M\"{o}bius function, this construction is equivalent to the one given by Sarnak in \cite{sarnak2011lectures}.

Let $\ell^{\infty}(\mathbb{N})$ be the algebra of all uniformly bounded, complex-valued functions on $\N$.
It is an abelian C*-algebra, with its unit element being the constant function $\textbf{1}$.
By Stone-Gelfand-Naimark theory,
any unital C*-subalgebra $\FAA$ of $\ell^{\infty}(\mathbb{N})$ is *-isomorphic to $C(X_\FAA)$,
the algebra of continuous functions on the space $X_\FAA$ of maximal ideals  of $\FAA$.
For every such $\FAA$, $X_\FAA$ is a compact Hausdorff space.
To define anqie,
we denote $\sigma_A$ to be the map on $\ell^{\infty}(\mathbb{N})$ given by
\[(\sigma_A f)(n)= f(n+1), \quad \forall f(n)\in  \ell^{\infty}(\mathbb{N}). \]
\begin{defi}[Anqie, \cite{ge2016, wei2022anqie}]
A unital C*-subalgebra $\FAA$ of $\ell^{\infty}(\mathbb{N})$ is called an
 \emph{anqie} if $\sigma_A(\FAA) \subseteq \FAA$.
\end{defi}
Every anqie induces an $\N$-dynamical system,
which is a continuous map $\widehat\sigma_A$ on $X_\FAA$ induced by $\sigma_A$.
Now we describe the map $\widehat\sigma_A$.
Denote by $ \hat n$ the multiplicative state (also known as a maximal ideal, hence a point) in $X_\FAA$ induced by a natural number $n\in \N$,
that is,
\[\hat n (f(\cdot)) = f(n), \quad \forall f\in \FAA.  \]
The mapping $n\mapsto \hat n$ maps $\N$ into a dense subset $\widehat{\N}$ of $X_\FAA$.
Denote by $\widehat\sigma_A$  the map on $\widehat{\N}$ induces by $\sigma_A$, i.e., $\widehat\sigma_A: \hat n \mapsto \widehat {n+1} $.
It can be proved that $\widehat\sigma_A$ can be continuously extended to a map from $X_\FAA$ to $X_\FAA$.
In other words,
$(X_\FAA, \widehat\sigma_A)$ is a topological $\N$-dynamical system.

We are interested in the case when an anqie is generated by one sequence.
In this context, the following is an alternative way to describe the dynamical system $(X_\FAA,
\widehat\sigma_A)$,
which is closely related to the construction given by Sarnak in \cite{sarnak2011lectures}.
Let $f(n)\in \ell^{\infty}(\mathbb{N})$,
and denote $\FAA$ to be the anqie generated by $f$.
Let $X_0= \overline{f(\N)}$ be the closure of the image of $f$.
Every element $g\in \FAA$ is a sequence in $X_0$,
so it is also a point in $X_0^\N$,
meaning that the sequence $g$ is treated as
\[ (g(0), g(1),g(2), ...) \in X_0^\N.\]
From this perspective,
$f$ and its shifts $\sigma_A^m (f)$ can be viewed as points in $X_0^\N$,
that is,
\[ \sigma_A^m (f) = (f(m), f(m+1), f(m+2), ...) \in X_0^\N.\]
The closure of the set $\{\sigma_A^m (f): m\in \N\}$ in $X_0^\N$ is a compact metric space, denoted as $X_f$.
And, the map $\sigma_A^m (f) \mapsto \sigma_A^{m+1} (f) $ extends uniquely to a continuous map on $X_f$,
which we again denoted as $\sigma_A$.
Then, $(X_f, \sigma_A)$ is also an $\N$-dynamical system with transitive point $f\in X_f$.
The dynamical system $(X_f, \sigma_A)$,
by Theorem 3.5 in \cite{wei2022anqie},
is isomorphic to the above construction $(X_\FAA, \widehat\sigma_A)$ as $\N$-dynamical systems.
And this isomorphism maps the point $\hat 0\in X_\FAA$ to the point $f \in X_f$.

It is natural to use ``entropy" to describe the complexity of a sequence.
The anqie entropy of a sequence is based formally on the topological entropy of the corresponding dynamical system.
Here, let us first recall the definition of topological entropy for an $\mathbb{N}$-dynamic systems $(X, T)$, where $X$ is a compact Hausdorff space and $T$ is a continuous map on $X$.
Suppose $\mathcal{U}$ and $\mathcal{V}$ are two open covers for $X$.
Denote by $\mathcal{U} \vee \mathcal{V}$ the open cover containing all intersections of elements from $\mathcal{U}$ and $\mathcal{V}$ (i.e., $\mathcal{U} \vee \mathcal{V}=\{A \cap B$ : $A \in \mathcal{U}, B \in \mathcal{V}\})$,
and by $\CNN (\mathcal{W})$ the minimal number of open sets in $\mathcal{W}$ that will cover $X$. Define
$$
\begin{aligned}
h(T, \mathcal{U}) & =\lim _n \frac{1}{n}\left\{\log \left(\CNN \left(\mathcal{U} \vee T^{-1}(\mathcal{U}) \vee \cdots \vee T^{-n+1}(\mathcal{U})\right)\right)\right\}, \\
h(T) & =\sup _{\mathcal{U}}\{h(T, \mathcal{U}): \mathcal{U} \text { is an open cover of } X\},
\end{aligned}
$$
$h(T)$ is called the \emph{(topological) entropy} of $T$.
Then, the anqie entropy is defined as follows.

\begin{defi}[Anqie entropy, \cite{ge2016}]
Suppose $\FAA \subseteq \ell^{\infty}(\mathbb{N})$ is an anqie.
Define the \emph{anqie
entropy} of $\FAA$,
denoted as $\AE(\FAA)$, to be the topological entropy of the dynamical system $(X_\FAA, \widehat\sigma_A)$.
If $\FAA$ is generated by a bounded sequence $\{f(n)\}_{n\in\N}$,
we also use $\AE(f)$ to denote $\AE(A)$, which is called the \emph{(anqie) entropy} of $f(n)$.
\end{defi}

The anqie entropy has many nice properties. Here are some of them that will be used in this paper.

\begin{lem}[{\cite[Lemma 3.3]{ge2016}}]
\label{some properties about anqie entropy}
Suppose $\mathcal{A}, \mathcal{B} \subseteq l^{\infty}(\mathbb{N})$ are anqies.
Then we have the following:
\begin{enumerate}
\item
$\AE(\mathcal{A}) \geq 0$;
\item
If $\mathcal{A}$ is a subanqie of $\mathcal{B}$, then $\AE(\mathcal{A}) \leq \AE(\mathcal{B})$;
and
\item
For any $f_1, \ldots, f_n \in l^{\infty}(\mathbb{N})$ and any polynomials $\phi_j \in \mathbb{C}\left[x_1, \ldots, x_n\right]$ with $1 \leq$ $j \leq m$, we have
$$
\AE\left(\phi_1\left(f_1, \ldots, f_n\right), \ldots, \phi_m\left(f_1, \ldots, f_n\right)\right) \leq \AE\left(f_1, \ldots, f_n\right).
$$
\end{enumerate}
\end{lem}

The following lemma is about anqie entropy and algebraic operations.
\begin{lem}[\cite{ge2016,wei2022anqie}]
\label{231210cor1}
Let $\{f(n)\}_{n\in\N}$ and $\{g(n)\}_{n\in\N}$ be two bounded sequences in $\CN$,
then
\begin{enumerate}
\item
    $\AE(f(n+1))= \AE(f(n))$, and $\AE(\overline{f(n)})= \AE(f(n))$;
\item
${\AE}(f(n))= {\AE}(cf(n))$ for any nonzero constant $c$;
\item
${\AE}(f(n)\pm g(n))\leq {\AE}(f(n))+{\AE}(g(n)) $;
and
\item
${\AE}(f(n) g(n))\leq {\AE}(f)+{\AE}(g) $.
\end{enumerate}
\end{lem}

We remark here that, there do exist two functions $f(n)$ and $g(n)$ with positive entropy such that ${\AE}(f(n)\pm g(n))= {\AE}(f)+{\AE}(g) $.
For example,
let $f(n)$ be a random sequence taking values from $\{0,1/2\}$,
and let $g(n)$ be a random sequence taking values from $\{0,1/4\}$ independently of $f(n)$.
Then $ f(n)+g(n)$ takes values from $\{0,1/4,1/2,1/4\}$ randomly,
and $ f(n)-g(n)$ takes values from $\{-1/4, 0, 1/4, 1/2\}$ randomly,
hence $\AE(f(n)+g(n))=\log 4=\AE(f(n))+\AE(g(n))$.

The next Lemma is about the lower semi-continuity of anqie entropy.
\begin{lem}[{\cite[Corollary 4.1]{ge2016}}]
\label{semicontinuity}
If $\{f_{N}(n)\}_{n\in\N}$ is a sequence of bounded arithmetic functions converging to $f(n)$ uniformly with respect to $n\in \mathbb{N}$, then $\liminf_{N\rightarrow \infty}\AE(f_{N})\geq \AE(f)$.
\end{lem}

We shall be interested when $f(n)$ has a finite range.
In this case,
its entropy is determined by the number of different $J$-blocks appearing in $f(n)$.
Specifically, let $\mathcal{B}_{J}(f)$ denote the set of all $J$-blocks occurring in $f(n)$, i.e.,
$$\mathcal{B}_{J}(f)=\{(f(n),f(n+1),\ldots,f(n+J-1)):n\geq 0\}.$$
Then the entropy of $f(n)$ equals
\begin{equation}\label{formula of finite range}
\lim_{J\rightarrow \infty}\frac{\log|\mathcal{B}_{J}(f)|}{J},
\end{equation}
where $|\mathcal{B}_{J}(f)|$ is the cardinality of the set $\mathcal{B}_{J}(f)$
(see \cite[Lemma 6.1]{wei2022anqie}).
Moreover, a $J$-block of the form
$$\left(f(mJ), f(mJ+1),\dots,f(mJ+J-1)\right)$$
for some $m\in\N$ is called a \emph{regular J-block} of $f(n)$.
A $J$-block that occurs infinitely many times in the sequence $f(n)$ is called an \emph{effective J-block} of $f(n)$.
A $J$-block $\left(\mathbf{b}(0), \mathbf{b}(1), \ldots, \mathbf{b}(J-1)\right)$ of $f(n)$ is called \emph{effectively regular} if there are infinitely many natural numbers $m$ such that
$$
\left(\mathbf{b}(0), \mathbf{b}(1), \ldots, \mathbf{b}(J-1)\right)=(f(m J),f(mJ+1), \ldots, f(m J+J-1)) .
$$
The set of all effectively regular (resp. effective/regular) blocks of $f(n)$ is denoted by $\mathcal{B}_J^{e, r}(f)$ (resp. $\mathcal{B}_J^e(f) $/$B_J^r(f)$.)
\begin{lem}[{\cite[Proposition 2.4]{GW}}]\label{lem:Jblock}
Let $\{f(n)\}_{n\in\N}$ be a sequence with finite range.
Then
    $$\AE(f)=
    \lim\limits_{J\rightarrow\infty}\frac{\log |B_J^r(f)|}{J}
    =\lim\limits_{J\rightarrow\infty}\frac{\log |B_J^{e,r}(f)|}{J}
    =\lim\limits_{J\rightarrow\infty}\frac{\log |B_J^{e}(f)|}{J}.
    $$
\end{lem}

The following result implies that any sequence of entropy $\lambda$ can be approached by functions with finite ranges and entropy at most $\lambda$.
\begin{lem}[{\cite[Theorem 1.8]{wei2022anqie}}]
\label{approach by finite ranges with zero anqie entropy}
Suppose $\{f(n)\}_{n\in\N}$ is a sequence taking values in the torus $\R/\Z$ with anqie entropy $\lambda$ $(0\leq \lambda<+\infty)$.
Then for any $N\geq 1$, there is a sequence $\{f_{N}(n)\}_{n\in\N}$ taking values in the torus $\R/\Z$,
such that $\AE(f_{N}) \leq \lambda$ and $\sup_{n\in \mathbb{N}}|f_{N}(n)-f(n)|\leq \frac{1}{N}$.
\end{lem}

For the convenience of our use,
we strengthen Lemma \ref{approach by finite ranges with zero anqie entropy} to Lemma \ref{231210lem1},
and the idea of the proof comes from that of \cite[Theorem 6.3]{wei2022anqie}.
For completeness, we provide its proof in \ref{appendixA}.
\begin{lem}\label{231210lem1}
Let $\{x(n)\}_{n\in\N}$ be a sequence taking values
in the torus $\R/\Z$.
Suppose $\AE(x(n))=\lambda$,
then for any $\epsilon>0$ and $N\ge 1$, there is a sequence $\{g_{N}(n)\}_{n\in\N}$ taking values in $\{0,\frac{1}{N},\frac{2}{N},..., \frac{N-1}{N}\}$
such that $\AE(g_{N}(n)) \le \lambda$ and $\sup_{n\in \mathbb{N}}|g_{N}(n)-x(n)|\le  \frac{1+\epsilon}{2N}$.
\end{lem}

\section{Entropic van der Corput’s theorem
}\label{240116sec1}
In this section, we prove Theorem \ref{231211thm1} and Corollary \ref{231210thm2}.
The main part of the proof of  Theorem \ref{231211thm1} is the following lemma.

\begin{lem}\label{231211lem1}
Let $\{x(n)\}_{n\in\N}$ be a sequence taking values
in the torus $\R/\Z$.
Then for every $d\ge 1$, the sequence $\{x(n+d) - x(n)\}_{n\in\N}$ has entropy greater than or equal to that of $\{x(n)\}_{n\in\N}$.
\end{lem}
\begin{proof}
Given $d\ge 1$, we denote $x_d(n)=x(n+d)-x(n)$.
Assume $\AE(x_d)=\lambda$.
We will prove the statement by constructing a sequence $\{g_N(n)\}_{n\in\N}$ so that $\lim\limits_{N\rightarrow \infty}\|g_N(n) - x(n)\|_\infty=0 $ and $\lim\limits_{N\rightarrow \infty} \AE(g_N(n))=\lambda$.

By Lemma \ref{231210lem1},
we can find a finite range function $f_N(n)$ taking values in $\{0,\frac{1}{N^2},\frac{2}{N^2},$ $..., \frac{N^2-1}{N^2}\}$, such that
$$\AE(f_N(n)) \le \lambda \quad \text{ and }\quad \sup_{n\in \mathbb{N}}|f_N(n)-x_d(n)|\le  \frac{2}{N^2}.$$
Denote $\mathcal{B}_{J}(f_N)$ to be the set of all $J$-blocks occurring in $f_N(n)$.
By \eqref{formula of finite range}, we have
\[ \lim_{J\rightarrow \infty} \frac{\log |\mathcal{B}_{J}(f_N) | }{J} \le \lambda.\]

Now for every $N\ge d$, we construct a function $g_N(n)$ which takes values in $\{0,\frac{1}{N^2},\frac{2}{N^2},..., \frac{N^2-1}{N^2}\}$.
First,
when
\begin{align}\label{231211eq1}
n\equiv r  \mod (dN)   \quad \text{ for some } \quad    0\le r<d,
\end{align}
define
$$g_N(n) =\frac{m}{N^2}\quad  \text{ if } \quad x(n)\in [\frac{m}{N^2}, \frac{m+1}{N^2} ).$$
In particular, $|g_N(n) - x(n)|\le \frac 1{N^2}$ in this case.
Second,
when $n$ is not of the form \eqref{231211eq1},
assume
$$n=n_1dN +qd+r \quad \text{ for some } \quad n_1\in \N,\; 1\le q<N \text{ and  } 0\le r< d,$$
define
\begin{align}\label{231211eq2}
g_N(n) = g_N(n_1dN+r) + \sum_{k=0}^{q-1} f_N(n_1dN+r+kd) \mod 1.
\end{align}
Since $f_N(n)$ takes values in $\{0,\frac{1}{N^2},..., \frac{N^2-1}{N^2}\}$,
so is $g_N(n)$.
Moreover, in case \eqref{231211eq2},
\begin{align*}
&|g_N(n)- x(n)| \\
= & |g_N(n_1dN+r) + \sum_{k=0}^{q-1} f_N(n_1dN+r+kd)- x(n)|\\
\le & |g_N(n_1dN+r) + \sum_{k=0}^{q-1} x_d(n_1dN+r+kd) - x(n_1dN +qd+r)|  \\
&+| \sum_{k=0}^{q-1} f_N(n_1dN+r+kd)  - \sum_{k=0}^{q-1} x_d(n_1dN+r+kd) |\\
\le & |g_N(n_1dN+r)  - x(n_1dN+r)|+ \sum_{k=0}^{q-1} |f_N(n_1dN+r+kd) - x_d(n_1dN+r+kd) |\\
\le & \frac 1{N^2}+ k \frac 2{N^2} \le \frac2N.
\end{align*}
Therefore, $\lim_{N\rightarrow \infty}\|g_N(n) - x(n)\|_\infty=0 $.

In the following, we estimate the entropy of $g_N(n)$.
Since $g_N(n)$ is a finite range function,
we can use Lemma \ref{lem:Jblock} to compute its entropy by counting its regular $J$-blocks,
that is,
by counting the number of different blocks of the form
\begin{align}\label{231211eq3}
(g_N(mJ), g_N(mJ+1),..., g_N(mJ+J-1) ),\quad m\in \N.
\end{align}
For simplicity,
we may assume $J=KdN, \;K\in\N$, is a multiple of $dN$ (since the limit in Lemma \ref{lem:Jblock} always exists).

To estimate the regular $J$-blocks, we consider the following partition of $\N$ by $J$-blocks of $f_{d} $.
For any block
$\mathbf{b}  \in \mathcal{B}_{J}(f_N)$,
let
\begin{align*}
\BAA_{\mathbf{b}} = \{m\in \N:\; (f_N(mJ) ,f_N(mJ+1), ... , f_N(mJ+J-1)) = \mathbf{b} \}.
\end{align*}
Then
\[\N = \bigsqcup_{\mathbf{b} \in  \mathcal{B}_{J}(f_N) } \BAA_{\mathbf{b}}\]
is a partition of $\N$.

We first count  the number of different blocks of the form
\begin{align}\label{231211eq4}
(g_N(mJ), g_N(mJ+1),..., g_N(mJ+J-1) ),\quad m\in \BAA_{\mathbf{b}}
\end{align}
for some given
\[ \mathbf{b} = (\mathbf{b}(0), \mathbf{b}(1),...,\mathbf{b}(J-1)) \in \mathcal{B}_{J}(f_N).\]
For $0\le j\le J-1$,
assume
\begin{align*}
j = j_1dN +qd+r \quad \text{ for some } \quad 0\le j_1<K,\; 1\le q<N \text{ and  } 0\le r< d.
\end{align*}
By \eqref{231211eq2},
\begin{align*}
g_N(mJ+j) = &g_N((mk+j_1)dN +qd+r) \\
=& g_N((mk+j_1)dN +r ) + \sum_{k=0}^{q-1} f_N((mk+j_1)dN+r+kd) \mod 1\\
=& g_N(mJ+j_1dN +r) + \sum_{k=0}^{q-1} \mathbf{b}(j_1dN+r+kd)\mod 1,
\end{align*}
where the last equation is because $n\in \BAA_{\mathbf{b}}$.
Therefore,
under the condition that $n\in \BAA_{\mathbf{b}}$,
the block
\[(g_N(mJ), g_N(mJ+1),..., g_N(mJ+J-1) )\]
is determined by the values
$$g_N(mJ+j_1dN +r), \quad 0\le j<K,\; 0\le r<d,$$
which has at most $N^{2Kd}$ choices (since each $g_N(mJ+j_1dN +r)$ has form $j/N^2$ for some $0\le j<N^2$).
Therefore, the number of different blocks in \eqref{231211eq4} is at most $N^{2\red{K}d}$.

Now,
since the vector $\mathbf{b}$ has $|\mathcal{B}_{J}(f_N) |$ different choices,
the number of different blocks of form \eqref{231211eq3} is at most
\[N^{2Kd} |\mathcal{B}_{J}(f_N)  |, \]
and hence
\begin{align*}
\AE(g_N) = & \lim_{J\rightarrow \infty} \frac{\log|\mathcal{B}^r_J(g_N)| }{J}\\
\le & \lim_{J\rightarrow \infty} \frac{1}{J} \log \left(N^{2Kd} |\mathcal{B}_{J}(f_N)  | \right) \\
= & \lim_{J\rightarrow \infty} \frac{1}{J} \log (N^{J/N}) +  \lim_{J\rightarrow \infty} \frac{1}{J}\log |\mathcal{B}_{J}(f_N)|  \\
=& \frac{1}{N}\log N + \AE(f_N).
\end{align*}
Since $\lim_{N\rightarrow \infty}\|g_N(n) - x(n)\|_\infty=0 $,
by Lemma \ref{semicontinuity}, we have
\begin{align*}
\AE(x(n)) \le \liminf_{N\rightarrow \infty} \AE(g_N) \le \liminf_{N\rightarrow \infty} \left(\frac{1}{N}\log N + \AE(f_N)\right) \le \lambda,
\end{align*}
which completes the proof.
\end{proof}

\begin{proof}[Proof of Theorem \ref{231211thm1}]
Given $d\ge 1$,
we denote $x_d(n)=x(n+d)-x(n)$,
and denote $\mathcal{A}$ (resp., $\mathcal{B}$) to be the anqie generated by $x(n)$ (resp., $x_d(n)$).
By Lemma \ref{231211lem1},
$\AE(x(n)) \le \AE(x_d(n))$.
On the other hand,
since the functions $n\mapsto x(n+d)$ and $n\mapsto x(n)$ are both in $\mathcal{A}$,
the function $x_d: n\mapsto x(n+d)-x(n) $ is in $\mathcal{A}$ as well.
Then by Lemma \ref{some properties about anqie entropy} (ii), $\AE(\mathcal{B}) \le \AE(\mathcal{A})$,
that is, $\AE(x_d(n)) \le \AE(x(n))$.
So the proof is completed.
\end{proof}

\begin{proof}[Proof of Corollary \ref{231210thm2}]
When $x(n)$ is a uniformly bounded sequence in $\R$ instead of $\R/\Z$,
we can multiply $x(n)$ by a small constant to make it in $[0,1/3)$ (which does not change the entropy of the sequence, see Lemma \ref{231210cor1},)
then Theorem \ref{231211thm1} implies that the sequence still has the same entropy as its difference.
When $x(n)$ takes values in $\CN$,
it has zero entropy if and only if both its real part and imaginary
part are deterministic.
Hence Theorem \ref{231211thm1} implies the statement.
\end{proof}

\section{A special case: entropy of $ p^n x \mod 1$}\label{240115sec1}
From now on let us consider the entropy of sequences $\{a_x(n)\}_{n\in \N}$,
where and in the following we denote $a_x(n):=a(n)x$ for an integer sequence $\{a(n)\}_{n\in \N}$ and a real number $x$.
We start from a special case
when $a(n) = p^n$ for some integer $p\ge 2$.
In this section, we show that
the (anqie) entropy of $a_x(n)$ can be computed by counting the blocks appearing in the expansion of $x$ to the base $p$.
This can be shown by determining the dynamical system of the sequence $a_x(n)$,
which will be a classical $\times p \mod1$ system on the orbit $\overline{\{p^n x\}}$ of $x$.
However, we will use an alternative method instead of describing the dynamical system because, in general,
it is difficult to determine the dynamical system of a sequence.

Without loss of generality we assume the expansion of $x$ to the base $p$ is $$ 0. c(1)c(2)c(3)\cdots,$$
that is,
\begin{align*}
x = \sum_{k=1}^\infty \frac{c(k)}{p^k}.
\end{align*}
Then
\begin{align*}
a_x(n) = p^n x = \sum_{k=1}^\infty \frac{c(k+n)}{p^k},\quad \forall n\in \N.
\end{align*}
Let $L$ be a given integer.
By Lemma \ref{231210lem1},
there is a sequence $g_L(n)$ taking values in $\{j/p^L \mod 1: 0\le j< p^L\}$ such that
\begin{align}
\sup_n |g_L(n) - a_x(n)| < \frac 1{p^{L}} \quad \text{ and } \quad \AE(g_L) \le \AE(a_x).
\end{align}
Let $J>0$.
When
\begin{align*}
(b_1, b_2, ..., b_{L+J})
\end{align*}
is a block in the expansion of $x$ to the base $p$,
meaning we can find $k$ with
\[(b_1, b_2, ..., b_{L+J}) = (c(k+1), c(k+2),..., c(k+L+J)), \]
then we see that
\[g_L(k+j) = \sum_{l=1}^L \frac{b_{l+j}}{p^l} + \frac{\xi_j}{p^L},\quad j=0,...,J-1,\]
for some appropriate choices of $\xi_j\in \{0,1\} $.
So we can build a map from $(L+J)$-blocks of $c(k)$ to $J$-blocks of $g_L$,
by mapping  $(b_1, b_2, ..., b_{L+J})$ to any such block $(g_L(k),..., g_L(k+J-1)) $.
The map may not be injective,
but it is easy to check that it is at most 2-1.
So $|\CBB_J(g_L(n))| \ge \frac12 |\CBB_{L+J}(c(k))|$,
and
\[\AE(a_x) \ge \AE(g_L) =\lim_{J\rightarrow \infty} \frac{\log |\CBB_J(g_L(n))| }{J} \ge \lim_{J\rightarrow \infty} \frac{\log  |\CBB_{L+J}(c(k))| }{J} = \lim_{J\rightarrow \infty} \frac{\log  |\CBB_{J}(x^{(p)})| }{J},
\]
where $\CBB_{J}(x^{(p)})$ denotes the set of length $J$ blocks in the expansion $x^{(p)}$ of $x$ to the  base $p$.

On the other hand,
define
\[f_L(n) = \sum_{l=1}^L \frac{c(l+n)}{p^l},\quad \forall n\in \N.\]
Then the sequence $f_L$ takes values from $\{j/p^L \mod 1: 0\le j< p^L\}$, and $
\sup_n |g_L(n) - a_x(n)| \le \frac 1{p^{L}}$.
Then $J$-blocks of $f_L$ are 1-1 correspond to $L+J$-blocks in $c(k)$,
hence
\[ \AE(f_L) = \lim_{J\rightarrow \infty} \frac{\log |\CBB_J(f_L(n))| }{J} = \lim_{J\rightarrow \infty} \frac{\log  |\CBB_{L+J}(c(k))| }{J} = \lim_{J\rightarrow \infty} \frac{\log  |\CBB_{J}(x^{(p)})| }{J}. \]
By Lemma \ref{semicontinuity},
$\AE(a_x) \le \liminf_{L\rightarrow \infty}\AE(f_{L}) $.
So we have proved the following proposition.
\begin{prop}\label{240115prop1}
Let $a(n) = p^n$ for some positive
integer $p \ge 2$,
$x$ be a real number and $x^{(p)} $ the expansion sequence of $x$ to the base $p$.
Denoting $a_x(n) = a(n) x$,
we have
\[\AE(a_x) = \lim_{J\rightarrow \infty} \frac{\log  |\CBB_{J}(x^{(p)})| }{J} \le \log p. \]
\end{prop}

\section{Dual entropy}
\label{section240103}
Proposition \ref{240115prop1} shows that,
the entropy of the sequence $\{p^n x \mod 1\}_{n\in \N}$ is equal to the entropy of the expansion of $x$ to the base $p$.
In particular,
$\{p^n x \mod 1\}_{n\in \N}$ have entropy $\log p$ for almost all real numbers $x$.
We believe the same property holds for every integer sequence,
which leads us to define the dual entropy.

\begin{defi}[Dual entropy]\label{240117defi1}
Let $\{a(n)\}_{n\in\N}$ be a sequence of integers.
If the sequences $a_x(n): = a(n)x \bmod 1 $ have the same entropy $\lambda$ for almost all real numbers $x$,
then we say $\{a(n)\}_{n\in\N}$ is a \emph{centered sequence},
and $\lambda$ is called the \emph{dual entropy} of $\{a(n)\}_{n\in\N}$, denoted as $\AE^*(\{a(n)\}_{n\in\N}) = \lambda$.
\end{defi}
It is natural to ask whether all integer sequences are centered.
We believe the answer is ``yes"
since if we replace ``having the same entropy" by ``uniformly distributed mod 1",
then it is a theorem of Weyl (see \cite{weyl4, weyl2,  kuipers2012uniform}): given an integer sequence $\{a(n)\}_{n\in\N}$,
the sequences $\{a_x(n)\}_{n\in \N}$ are uniformly distributed mod 1 for almost all real numbers $x$.
Moreover,
the following lemma (as well as Theorem \ref{240116prop1}) provides a vast class of centered sequences.
Here we define recursively the difference operator $\Delta^k$ on a sequence $x(n)$ by $\Delta^1 x(n)=\Delta x(n)=x(n+1)-x(n)$ and $\Delta^k x(n)=\Delta\left(\Delta^{k-1} x(n)\right)$ for $k \geq 2$.

\begin{lem}\label{240117lem1}
Let $\{a(n)\}_{n\in\N}$ be a sequence of integers such that $\Delta a(n)$ is bounded uniformly.
Then, for any irrational $x$, the sequences $a_x(n):= a(n)x \mod 1 $ have the same entropy,
which is equal to the entropy of $\Delta a(n)$.
\end{lem}
\begin{proof}
Since $\Delta a(n)$ is bounded,
we may assume  $-L\le a(n+1)-a(n) \le L$ for every $n$.
The sequence $d(n):= \Delta a(n)$ takes values in $\{-L,...,L\}$,
hence, it is a finite range function.
Let $\lambda = \AE(d(n))$.
In the following, we show that,
when $x$ is not a rational number,
the sequence $a_x(n) $ has entropy $\lambda$.

Consider the sequence $(\Delta a_x):= a_x(n+1) - a_x(n) = d(n) x \mod 1$.
The sequence $\Delta a_x$ takes values in $\{-xL,-xL+x...,xL\}\mod 1$.
Moreover,
since $x$ is irrational, the numbers $-xL,...,xL \mod 1$ are distinct.
So the map $\phi: j\mapsto (jx \mod 1)$ is a bijection from $\{-L,-L+1,...,L\} $ to $\{-xL,...,xL\}\mod 1$,
and it extends to a bijection from the $J$-blocks of $d(n)$ to the $J$-blocks of $\Delta a_x(n)$.
Hence $\AE (\Delta a_x(n)) = \lambda$.
By Theorem \ref{231211thm1},
$\AE (a_x(n)) = \lambda$.
\end{proof}
Note the Theorem \ref{240109thm1} is a corollary of Lemma \ref{240117lem1}.
The above proof also works for a higher-ordered difference statement:
assume $a(n)$ is a sequence of integers with its  $k$-th difference $\Delta^k a(n)$ uniformly bounded,
then $a_x(n):= a(n)x \mod 1 $ have the same entropy for every irrational $x$.
To see this,
in the above proof we let $d(n):=\Delta^k a(n)$ instead of $d(n) = a(n+1)-a(n)$,
and consider $\Delta^k a_x(n)$ instead of $\Delta a_x(n)$,
then the argument is derived by the same proof.

However,
when $a(n)$ grows exponentially,
the above method is no longer applicable.
In this case, we prove Theorem \ref{240112thm1},
which gives a lower bound of $\AE (a_x(n))$.
Its proof mainly contains two parts.
First is to show the set of $x$ with $\AE(a_x(n))<\log p$ behaves like a Cantor set,
and second is to estimate the measure of the set by its ``dimension" in some sense.
\begin{thm*}[Theorem \ref{240112thm1}, restated]
Let $a(n)$ be an increasing sequence of distinct integers such that
$$\lim_n a(n+1)/a(n)= p,$$
where $p$ may not be an integer.
Then for almost all real numbers $x$, the sequence $a_x(n) = a(n)x \mod 1 $ has entropy at least $\log p$.
\end{thm*}
\begin{proof}
When $p=1$ the statement is trivial.
In the following we assume $p>1 $.
Since $a(n)x = a(n)(x+k) \mod 1 $ for any integer $k$ and $n$, we only need to deal with the case when $x\in [0,1)$.

Given $\epsilon>0$.
Let $N $ be large enough such that $p^N>\frac1{\epsilon}$,
and we denote $M$ to be the largest integer such that
\begin{align}\label{240113eq1}
\lfloor p^N - 1\rfloor \left(\frac{1 }{M} - \frac{1 }{M^2} \right) > 1+\frac 1M.
\end{align}
We may also assume $N $ is large enough so that $M> \frac1{\epsilon}$.
We also let $L$ be large enough such that,
when $n\ge L$,
\begin{align*}
\left|\frac{a(n+N)}{ a(n)} - p^N \right|<\epsilon',
\end{align*}
where $\epsilon'$ is a small parameter to be chosen later.
By Lemma \ref{231210lem1}, for every $x\in [0,1)$ there is a function $f_{x,M}(n)$ taking values in $\{0,\frac{1}{M},\frac{2}{M},..., \frac{M-1}{M}\}$ such that
\begin{align}\label{240112eq9}
\AE(f_{x,M}(n)) \le \AE(a_x(n)) \quad \text{ and } \quad \sup_{n\in \mathbb{N}}|f_{x,M}(n)-a_x(n)|\le  \frac{1 }{2M} + \frac{1}{M^2}.
\end{align}
We will show that for almost every $x\in [0,1)$,
$f_{x,M}(n)$ contains enough different $MJ$-blocks for every $J\ge 1$,
so that $ \AE(f_{x,M} )$ is large.

Let
\[ \mathbf{b} = (\mathbf{b}(0), \mathbf{b}(1),...,\mathbf{b}(J-1)) \in \left\{0,\frac{1}{M},\frac{2}{M},..., \frac{M-1}{M}\right\}^J. \]
We say $\mathbf{b}$ is a \emph{$(J,N)$-block} of $f_{x,M}(n)$ (at position $k\in \N$) if
\[(f_{x,M}(kNJ), f_{x,M}(kNJ+N) , f_{x,M}(kNJ+2N) , ..., f_{x,M}(kNJ+JN)) = \mathbf{b}.\]
We first construct a set $\CSS_\mathbf{b}\subseteq [0,1)$ such that, for every $x$ in $\CSS_\mathbf{b}$, $f_{x,M}(n)$ must (independent of the choice of $f_{x,M}(n)$ for $x$) contain the block $\mathbf{b}$ as its $(J,N)$-block.

Let $k\ge 0$.
If $x$ satisfies that, for every $j=0,1,...,J-1$,
\begin{align}\label{240112eq10}
a_x(kNJ+jN)  \in  \left(\mathbf{b}(j) -\frac{1 }{2M} + \frac{1}{2M^2}, \;\mathbf{b}(j) +\frac{1 }{2M} - \frac{1}{2M^2}  \right),
\end{align}
then from \eqref{240112eq9} we have that $\mathbf{b}$ must be the $(J,N)$-block of $f_{x,M}(n)$ at position $k$.
And \eqref{240112eq10} is equivalent to that, for every $0\le j\le J-1$,
\[a(kNJ+jN)x \in  \left(\mathbf{b}(j) -\frac{1 }{2M} + \frac{1}{2M^2}, \;\mathbf{b}(j) +\frac{1 }{2M} - \frac{1}{2M^2}  \right) + \Z,\]
or equivalently,
\begin{align}
x \in  \left(\frac{m+\mathbf{b}(j) -\frac{1 }{2M} + \frac{1}{2M^2}}{a(kNJ+jN)}, \;\frac{m+\mathbf{b}(j) +\frac{1 }{2M} - \frac{1}{2M^2}}{a(kNJ+jN)}  \right)
\end{align}
for some integer $0\le m\le a(kNJ+jN)-1 $.
Therefore,
if
\begin{align}
x \in S_k:= \bigcap_{j=0}^{J-1} \;\; \bigcup_{m=0}^{a(kNJ+jN)-1 } \left(\frac{m+\mathbf{b}(j) -\frac{1 }{2M} + \frac{1}{2M^2}}{a(kNJ+jN)}, \;\frac{m+\mathbf{b}(j) +\frac{1 }{2M} - \frac{1}{2M^2}}{a(kNJ+jN)}  \right) ,
\end{align}
then $\mathbf{b}$ must be a $(J,N)$-block of $f_{x,M}(n)$.

Define
\begin{align*}
\CSS_\mathbf{b} = \bigcup_{k=L}^\infty S_k.
\end{align*}
Then for every $x$ in $\CSS_\mathbf{b}$, say $x\in S_k$,
then $\mathbf{b}$ is the $(J,N)$-block of $f_{x,M}(n)$ at position $k$.
Next we estimate the measure of $\CSS_\mathbf{b}$ by estimating the measure of its complement
\begin{align*}
\CSS_\mathbf{b}^c = \bigcap_{k=L}^\infty S_k^c.
\end{align*}

Denote the interval
\[A(k,m,i):= \left(\frac{m+\frac {i}M -\frac{1 }{2M} + \frac{1}{2M^2}}{a(kN)}, \;\frac{m+\frac {i}M +\frac{1 }{2M} - \frac{1}{2M^2}}{a(kN)}  \right). \]
Here and in the following, for convenience, when $m=i=0$, we denote
\[\left(\frac{ -\frac{1 }{2M} + \frac{1}{2M^2}}{*}, \;\frac{ \frac{1 }{2M} - \frac{1}{2M^2}}{*}  \right) = \left[0, \;\frac{ \frac{1 }{2M} - \frac{1}{2M^2}}{*}  \right) \bigcup \left( *+ \frac{ -\frac{1 }{2M} + \frac{1}{2M^2}}{*},\; * \right).\]
Let $\epsilon'$ be small enough so that,
when $ k\ge L$,
\begin{align}\label{240112eq11}
\frac{a(kN)}{ a((k-1)N)} \left(\frac{1 }{M} - \frac{1 }{M^2} \right) > 1+\frac 1M.
\end{align}
Now we claim that, when $k \ge L$ and $(y,z)\subseteq [0,1)$ is an open interval with size at least
\begin{align}\label{240112eq12}
 \frac1{a((kJ-1)N)} \left(\frac{1 }{M} - \frac{1 }{M^2} \right),
\end{align}
the intersection $S_k \cap (y,z)$ contains an entire open interval $A(kJ+J-1,m',i')$ for some appropriate $m',i'$.
Indeed,
by \eqref{240112eq11} we have
\[ \left(\frac{ * }{a((kJ)N) } ,\;  \frac{ * + 1+\frac 1M}{a((kJ)N) }\right)\subseteq (y,z), \]
hence there must be some $m_0$ such that
\[A(kJ, m_0, \mathbf{b}(0)) \subseteq   (y,z).\]
Similarly,
by \eqref{240112eq11} we have
\[ \left(\frac{ * }{a((kJ+1)N) } ,\;  \frac{ * + 1+\frac 1M}{a((kJ+1)N) }\right)\subseteq A(kJ, m_0, \mathbf{b}(0)), \]
hence there must be some $m_1$ such that
\[A(kJ+1, m_1, \mathbf{b}(1)) \subseteq   A(kJ, m_0, \mathbf{b}(0)).\]
We continue this process until we find  appropriate $m_0 , m_1 ,..., m_{J-1}$ such that
\[A(kJ+J-1, m_{J-1}, \mathbf{b}(J-1)) \subseteq \cdots \subseteq A(kJ, m_0, \mathbf{b}(0)) \subseteq   (y,z).\]
It is easy to check that $A(kJ+J-1, m_{J-1}', \mathbf{b}(J-1)) \subseteq S_k$,
hence
\begin{align}\label{240112eq13}
A(kJ+J-1, m_{J-1}', \mathbf{b}(J-1)) \subseteq S_k \cap (y,z)
\end{align}
as we claimed.

To estimate the size of $\CSS_\mathbf{b}^c$,
we denote
\[T_K = \bigcap_{k=L}^K S_k^c \]
and show their measures $|T_K|$ converges to 0 as $K\rightarrow \infty$.
Clearly, $T_K$ is a disjoint union of finitely many intervals.
We can divide these intervals into smaller intervals, say $(y_1, z_1)$, ..., $(y_q,z_q)$ and $(y_1', z_1')$, ..., $(y_r',z_r')$,
such that the total measure of $(y_1', z_1') \cup \cdots \cup (y_r',z_r')$ is smaller than $|T_K|/2$,
and the intervals $(y_1, z_1)$, ..., $(y_q,z_q)$ have a same length
\begin{align}\label{240112eq14}
= \frac1{a((K_1J-1)N)} \left(\frac{1 }{M} - \frac{1 }{M^2} \right)
\end{align}
for some large integer $K_1$.
Since this length satisfies the condition \eqref{240112eq12},
from \eqref{240112eq13} we have that for $i=1,...,q$,
\begin{align*}
|S_{K_1} \cap (y_i,z_i)| \ge & \frac{1}{ a((K_1J+J-1)N)} \left(\frac{1 }{M} - \frac{1 }{M^2} \right)\\
= & \frac{a((K_1J-1)N)}{ a((K_1J+J-1)N)} |  (y_i,z_i)| \\
\ge & |  (y_i,z_i)| \left( \frac{M(M+1)}{M-1}\right)^J,
\end{align*}
where the last equality comes from \eqref{240112eq11}.
Denote $\theta = \left( \frac{M(M+1)}{M-1}\right)^J$.
Then
$$|S_{K_1} \cap T_K| \ge \frac \theta 2 |T_K|,$$
and hence
\[ |T_{K_1} | \le |S_{K_1}^c \cap T_K| \le \left(1-   \theta/2 \right) |T_K|.\]
Since $\theta>0$,
we have that
\[\lim_{K\rightarrow \infty} |T_K| =0.\]
Hence $\CSS_\mathbf{b}^c$ is a zero measure set.
Therefore the set
\[\CSS : = \bigcap_{\mathbf{b}}\CSS_\mathbf{b} \]
contains almost every $x\in \R/\Z$,
where the intersection runs over all possible blocks $\mathbf{b}$ (of all lengths) in $\left\{0,\frac{1}{M},\frac{2}{M},..., \frac{M-1}{M}\right\}$.
By the definition of $\CSS_\mathbf{b}$ we have that,
for every $x\in \CSS$,
$f_{x,M}(n)$ contains every $J$-block as its $(J,N)$-block.
Therefore $f_{x,M}(n)$ has at least $M^J$ different regular $(JN)$-blocks.
By Lemma \ref{lem:Jblock},
\begin{align*}
\AE(f_{x,M})  =  \lim\limits_{J\rightarrow\infty}\frac{\log |B_{NJ}^r(f_{x,M})|}{NJ}
\ge   \lim\limits_{J\rightarrow\infty}\frac{J \log M}{NJ}
=  \frac{\log M}{N}.
\end{align*}
By the condition \eqref{240112eq9},
\begin{align*}
\AE(a_x(n)) \ge \sup_M \AE(f_{x,M}) \ge \sup_M \frac{\log M}{N}= \log p,
\end{align*}
where the last equation comes from the definition \eqref{240113eq1}.
\end{proof}
The above proof (with some modifications) also gives a slightly stronger statement:
if $\liminf_n a(n+1)/a(n)= p$, then for almost all real numbers $x$, the sequence $a_x(n)$ has entropy at least $\log p$.
On the other hand,
even with the condition $\lim_n a(n+1)/a(n)= p$,
there still exist sequences $a(n)$ such that,
the entropy of $a_x(n)$ is strictly greater than $\log p$ for every irrational $x$.
In the following, we give an example.
\begin{exm}\label{240116exm1}
In this example,
for every $L>0$ and integer $p>2$,
we will construct a increasing sequence $a(n)$ with $\lim_n a(n+1)/a(n)= p$,
such that $\AE(a_x(n))>\log p+L$ for every irrational $x$.

Let $p'> p^2$ with $\log p'- 2\log p>L$,
and $d(n)$ be a sequence taking values from $\{0,1,...,p'-1\}$ with $\AE(d(n)) = \log p'$.
Let $c(n) $ be the sequence such that $c(0)=0$ and $\Delta c(n) = d(n)$.
Then $c(n)\le np'$.
Denote $a'(n) = p^n$, and define
\[a(n) = a'(n) + c(n).\]
So $a_x(n) = a'_x(n) + c(n)x$ for any $x$,
and $$\lim_n a(n+1)/a(n) = p.$$
Moreover,
for every irrational $x$,
$\AE(a'_x) \le \log p$ (see Proposition \ref{240115prop1}),
and by Theorem \ref{231211thm1}, $\AE(c_x(n)) = \AE(xd(n)) = \log p'$.
Hence by Lemma \ref{231210cor1} (iii),
\begin{align}\label{240117eq1}
\AE(a_x(n)) \ge \AE(c_x(n)) - \AE(a'_x) \ge \log p'-\log p> L.
\end{align}
\end{exm}

\section{Properties of dual entropy}\label{section_twoentropy}

Recall that a sequence $a(n)$ of integers is called \emph{centered} if $a_x(n)$ have the same entropy for almost all $x$,
and the common entropy is called the \emph{dual entropy} and denoted as $\AE^*(\{a(n)\}_{n\in\N}) $.
By the properties of (anqie) entropy,
it is not hard to check the following properties of dual entropy.
\begin{prop}
Let $a(n)$ and $b(n)$ be two centered sequence of integers,
then we have the following:
\begin{enumerate}
\item $\AE^*(a(n)) \ge 0$;
\item
    $\AE^*(a(n+1))= \AE^*(a(n))$, and $\AE^*(\overline{a(n)})= \AE^*(a(n))$;
\item
${\AE^*}(a(n))\ge {\AE^*}(ca(n))$ for any nonzero integer $c$;
and \item if $a(n)\pm b(n)$ is centered, then
${\AE^*}(a(n)\pm b(n))\leq {\AE^*}(a(n))+{\AE^*}(b(n)) $;
\item if $\AE^*(b(n))=0$,
then $a(n) \pm b(n)$ is centered and $\AE^*(a(n) \pm b(n)) = \AE^*(a(n))$.
\end{enumerate}
\end{prop}

Moreover,
Theorem \ref{231211thm1},
Lemma \ref{240117lem1} and Theorem \ref{240112thm1} (with their proofs) give the following three propositions, respectively.

\begin{prop}
A sequence $a(n)$  of integers is centered if and only if its difference $\Delta a(n)$ is centered,
and in this case, they have the same dual entropy.
\end{prop}

\begin{prop}
Let $\{a(n)\}_{n\in\N}$ be a sequence of integers such that $\Delta^k a(n)$ is bounded uniformly for some $k\ge 1$.
Then $a(n)$ is centered, and
\[\AE^*(a(n)) = \AE( \Delta^k a(n)).\]
\end{prop}

\begin{prop}
If $a(n)$ is a centered sequence such that  $\lim_n a(n+1)/a(n)= p$, then  $\AE^*(a(n)) \ge \log p$.
\end{prop}

Since dual entropy is another kind of entropy defined for subsets of $\N$,
it is natural to compare it with anqie entropy.
For a given subset $A\subseteq \N$,
let $a(0)<a(1)<\cdots$ be all the elements in $A$,
then we can describe its complexity by two different notions of entropy:
one is the dual entropy of $a(n)$,
if it exists;
the other one is the entropy of the characteristic function $1_A(n)$.
In the following, we will also denote $\AE^*(A)$ by the dual entropy of $a(n)$ if there is no confusion.
Then it is natural to ask whether these two entropy, $\AE^*(A)$ and $\AE(1_A)$, coincide.
We will use a specific example to negatively answer this question.

Let $\CQQ$ be the set of square-free numbers,
meaning that a natural number $m\in \CQQ$ if and only if $m$ does not have square factors other than 1.
It is an important research object in number theory since $\CQQ$ is the support of the M\"{o}bius function $\mu(n)$.
Sarnak \cite{sarnak2011lectures} proved that
$$\AE(1_\CQQ)
=\AE(\mu^2(n))
= \frac{6}{\pi^2}\log2 .$$  
On the other hand, in Proposition \ref{240116prop1}, we will prove the dual entropy of $\CQQ$ is infinite.
So the entropy of $1_\CQQ$ and the dual entropy of $\CQQ$ are different.

To estimate the dual entropy $\CQQ$,
we need to describe the blocks appearing in $1_\CQQ$.
It was proved in \cite{sarnak2011lectures} that
a $\{0,1\}$-block appears in the sequence $1_\CQQ$ if and only if it is so-called admissible.
Here, we provide the definition.
\begin{defi}[Admissible block]
Call a subset $A$ of $\mathbb{N}$ \textbf{admissible} if its reduction $\bmod$ $p^2$ doesn't cover all the residue classes $\bmod  p^2$ for every prime $p$.
A finite block in $\{0,1\}$ is called an \textbf{admissible block} if its support is an admissible set.
\end{defi}
For example,
the block
\[(1,0,1,1,0,1)\]
is not admissible since its support is $\{1,3,4,6\}$, whose reduction
mod 4 is $\{1,2,3,4\}$ and covers all the residue classes mod 4.
And the block
\[(1,0,1,0,0,1)\]
is admissible.

Sarnak showed that the number of $J$-blocks of $1_\CQQ$ is $2^{\; 6J/\pi^2 + o(J)}$.
In some sense, what we need to do is to estimate the number of ``blocks" appearing in $\CQQ$.
For this purpose,
we first prove the following Lemma,
which roughly says that there are many ``blocks" in $\CQQ$.

\begin{lem}\label{240112lem1}
For any irrational number $x$ and open intervals $(y_1,z_1)$,..., $(y_J,z_J)$ in $\R/\Z$ ($J\ge 0$),
there is an admissible block of the form
\begin{align}\label{240112eq1}
(1,\underbrace{0,...,0}_{d_1-1},1,\underbrace{0,...,0}_{d_2-1},1,...1, \underbrace{0,...,0}_{d_J-1},1)
\end{align}
such that
\begin{align}\label{240112eq2}
d_jx\in (y_j,z_j),\quad j=1,2,...,J.
\end{align}
\end{lem}
\begin{proof}
We construct the admissible block by induction on $J$.
When $J=0$, there is nothing to prove.
Assume for $J\ge 0$ we have constructed an admissible block of the form \eqref{240112eq1}
such that condition \eqref{240112eq1} is satisfied.
By the definition of admissible block, the set
\[A_{J}= \left\{0,d_1,d_1+d_2, d_1+d_2+d_3,..., \sum_{j=1}^{J}d_{j}\right\}\]
is admissible.
Our admissible block for the case $J+1$ will be of the form
\begin{align}\label{240112eq5}
(1,\underbrace{0,...,0}_{d_1-1},1,\underbrace{0,...,0}_{d_2-1},1,...1, \underbrace{0,...,0}_{d_J-1},1,\underbrace{0,...,0}_{d_{J+1}-1},1 )
\end{align}
for some  appropriate $d_{J+1} $.
To make \eqref{240112eq5} an admissible block,
by definition, we only need to make the set
\[A_{J+1}= \left\{0,d_1,d_1+d_2, d_1+d_2+d_3,..., \sum_{j=1}^{J+1}d_{j}\right\}\]
be admissible,
meaning that its reduction $\bmod$ $p^2$ doesn't cover all the residue classes $\bmod$ $  p^2$ for every prime $p$.
When $ p^2>J+1$, clearly $A_{J+1}$ $\bmod$ $p^2$ cannot cover all the residue classes $\bmod$  $p^2$.
For $ p^2 \le J+1$,
we choose $d_{J+1}$ so that
\begin{align}\label{240112eq6}
d_{J+1} + \sum_{j=1}^J d_j \equiv 0 \mod p,\quad \forall p \le \sqrt{J+1}.
\end{align}
For every $d_{J+1} $ satisfying \eqref{240112eq6}, since $0\in A_J$, we have
\[A_{J+1} \bmod p^2 = A_J \bmod p^2,\quad \forall p \le \sqrt{J+1}, \]
and since $A_J$ is admissible, we have that $A_{J+1}$ is also admissible.

Finally we choose  appropriate $d_{J+1} $ satisfying \eqref{240112eq6} such that $d_{J+1}x\in (y_{J+1},z_{J+1})$.
The numbers $d_{J+1} $ satisfying \eqref{240112eq6} form an arithmetic progression of infinite length,
so the values $d_{J+1} x $ (recall that $x$ is irrational) are dense in $\R/\Z$.
So there exists $d_{J+1} $ satisfying \eqref{240112eq6} such that $d_{J+1}x\in (y_{J+1},z_{J+1})$,
 which gives a block needed.
\end{proof}

\begin{thm}\label{240116prop1}
Consider the set of square-free numbers $\CQQ$.
Denote $a(0)<a(1)<\cdots$ to be all the elements in $\CQQ$, such that for every $a\in \CQQ$, $\mu^2(a)=1$.
The dual entropy of $\{a(n)\}_{n\in\N}$ is infinite.
\end{thm}
\begin{proof}
To prove the dual entropy of $\{a(n)\}_{n\in\N}$ is infinite, we prove that for every irrational number $x$, the entropy of $a_x(n):=a(n)x$ mod $1$ is infinite.
Given any irrational number $x$. Denote $\Delta a_x(n)=a_x(n+1)-a_x(n)$. In the following, we estimate the entropy of $\Delta a_x(n)$ for any irrational number $x$.

By Lemma \ref{231210lem1}, for any $N\in\N$, we can find a finite range function $f_N(n)$ taking values in $\{0, \frac1 N, \frac 2 N, \dots, \frac{N-1}{N}\}$, such that
$$\AE(f_N)\leq \AE(\Delta a_x)
\text{ and }
\sup\limits_{n\in\N}|f_N(n)-\Delta a_x(n)|\leq \frac{1.5}{2N}.$$

Given any $J\in\N$ and any choose of $i_1, i_2,\dots,i_J\in\{1,2, \dots,N\}$.
Let $(y_j, z_j)=(\frac{i_j}{N}-\frac{0.1}{2N}, \frac{i_j}{N}+\frac{0.1}{2N}), j=1,2,\dots,J$.
By Lemma \ref{240112lem1},
there is an admissible block of the form
\begin{align}\label{240112eq7}
(1,\underbrace{0,...,0}_{s_1-1},1,\underbrace{0,...,0}_{s_2-1},1,...1, \underbrace{0,...,0}_{s_J-1},1)
\end{align}
such that $s_jx\in (y_j,z_j),\;j=1,2,...,J.$
Therefore, $\Delta a_x$ has a $J$-block
$$\left( \Delta a_x(m+1),..., \Delta a_x(m+J)\right) = \left(s_{1}x, s_{2}x,\dots,s_{J}x\right)$$
such that $\Delta a_x(m+j)\in\left(\frac{i_j}{N}-\frac{0.1}{2N},\frac{i_j}{N}+\frac{0.1}{2N}\right), \;j=1,2,\dots,J$.
Since $f_N(n)$ takes values in $\{0, \frac{1}{N},\dots,\frac{N-1}{N}\}$ and $\sup\limits_{n\in\N}|f_N(n)-\Delta a_x(n)|\leq \frac{1.5}{2N}$, we have
$$f_N(m+j)\in\left(\frac{i_j}{N}-\frac{0.8}{N},\frac{i_j}{N}+\frac{0.8}{N}\right),\quad j=1,2,\dots,J.$$
That is, we must have $f_N(m+j)=\frac{i_j}{N},\;j=1,2,\dots,N$.
Thus, for any choose of $i_1, i_2,\dots,i_J\in\{1,2, \dots,N\}$, there exists a $J$-block in $f_N$, such that
$$\left(f_N(m+1),f_N(m+2),\dots,f_N(m+J)\right)=\left(\frac{i_1}{N},\frac{i_2}{N},\dots,\frac{i_J}{N}\right),$$
and hence
\begin{align*}
\AE(\Delta a_x)
\geq  \AE(f_N)= \lim_{J\rightarrow \infty} \frac{\log|\mathcal{B}^r_J(f_N)| }{J}
\geq  \lim_{J\rightarrow \infty} \frac{1}{J} \log N^J
= \log N.
\end{align*}
Since $N$ can be arbitrarily large,
we have $\AE(\Delta a_x) = \infty$,
hence by Theorem \ref{231211thm1},
$\AE( a_x) = \infty$, for every irrational $x$.
Therefore, the dual entropy of $\{a(n)\}_{n\in\N}$ is infinite.

\end{proof}

From Proposition \ref{240116prop1} (as well as $\AE(1_\CQQ) = 6/\pi^2 \log 2$), we see that
the dual entropy and anqie entropy of the same subset may be different.
In the following theorem, we show that these two entropy will not be too different if the difference of the subset is bounded.
\begin{thm}\label{240113thm1}
Let $A\subseteq \N$ be a set of natural numbers,
and $a(0)<a(1)<\cdots$ be all the elements in $A$,
such that $\sup_n |a(n+1)-a(n)|\le L$ for some $L>0$.
For every irrational $x$, denote $a_x(n):=a(n)x$, then
\[\AE(1_A) \le \AE(a_x) \le L\cdot \AE(1_A). \]
In particular,
$1_A(n)$ has entropy zero if and only if $\{a(n)\}_{n\in\N}$ is centered and has dual entropy  zero.
\end{thm}
\begin{proof}
We first prove the second inequality,
that is,
\begin{align}\label{240115eq6}
\AE(a_x) \le L\cdot \AE(1_A).
\end{align}

Note that
$x$ is irrational.
Let
\[\textbf{b} = (b_1x, b_2x,... b_J x) \mod 1\]
be any $J$-block of $\Delta a_x(n) $,
then $1_A$ must have a block $\widehat{\textbf{b}} $ with length $LJ$ such that
\begin{align*}
\widehat{\textbf{b}} = (1,\underbrace{0,...,0}_{b_1-1},1, \underbrace{0,...,0}_{b_2-1},1,...1, \underbrace{0,...,0}_{b_J-1},1, * ****)
\end{align*}
Hence there is an injective map from $\mathcal{B}_{J}(\Delta a_x)$ to $\mathcal{B}_{LJ}(1_A) $,
and  $ |\mathcal{B}_{J}(\Delta a_x) | \le |\mathcal{B}_{LJ}(1_A)|$.
(When
$x$ is rational,
the map may not be well-defined,
since each $\textbf{b}$ may correspond to multiple $\widehat{\textbf{b}} $.
But we still have  $ |\mathcal{B}_{J}(\Delta a_x) | \le |\mathcal{B}_{LJ}(1_A)|$.)
Hence by \eqref{formula of finite range} we have
\begin{align*}
\AE(\Delta a_x) = \lim_{J\rightarrow \infty} \frac{\log|\mathcal{B}_{J}(\Delta a_x)|}{J} \le  \lim_{J\rightarrow \infty} \frac{\log|\mathcal{B}_{LJ}(1_A)|}{J} = L\cdot \AE(1_A),
\end{align*}
which gives \eqref{240115eq6}.

Next we prove \begin{align}\label{240304eq1}
\AE(1_A)\le \AE(a_x).
\end{align}
Consider an $(LJ)$-block $\widehat{\textbf{b}}$ of $1_A(n)$ starting with ``1", say
\begin{align}\label{240114eq1}
\widehat{\textbf{b}} = (1,\underbrace{0,...,0}_{b_1-1},1, \underbrace{0,...,0}_{b_2-1},1,...,1, \underbrace{0,...,0}_{b_K-1}),
\end{align}
with $b_1,...,b_K\ge 1$ and $b_1+\cdots +b_K=LJ$.
The set of all such blocks of $1_A(n)$ is denoted by $\CBB_{LJ}^1(1_A)$.
We will show $|\CBB_{LJ}^1(1_A)|\le (|\CBB_{J}(\Delta a_x)|+1)^L$, by constructing an injective map from $|\CBB_{LJ}^1(1_A)|$ to $(\CBB_{J}(\Delta a_x) \cup \{\text{NULL}\})^L$.

For every block $\widehat{\textbf{b}}$ in \eqref{240114eq1},
We group $b_1,...,b_K$ into subsets, each subset having exactly $J$ numbers except the last one,
and get blocks
\begin{align*}
\widehat{\textbf{b}}_1 =& (1,\underbrace{0,...,0}_{b_1-1},1,...,1, \underbrace{0,...,0}_{b_J-1},1),\quad  \widehat{\textbf{b}}_2 = (1,\underbrace{0,...,0}_{b_{J+1}-1},1,...,1, \underbrace{0,...,0}_{b_{2J}-1},1) , \\ \cdots,\quad \;
\widehat{\textbf{b}}_m = &(1,\underbrace{0,...,0}_{b_{(m-1)J+1}-1},1,...,1, \underbrace{0,...,0}_{b_{K-1}-1}, 1).
\end{align*}
Note that for each $0\le i\le m-1$ we added an extra ``1" at the end of $\widehat{\textbf{b}}_i$,
and we deleted the zeros at the end of $\widehat{\textbf{b}}_m$.
Each $\widehat{\textbf{b}}_i$, $0\le i\le m$, corresponds to a $J$-block of $\Delta a_x(n) $, for instance,
\begin{align*}
(1,\underbrace{0,...,0}_{b_1-1},1,...,1, \underbrace{0,...,0}_{b_J-1},1) \;\; \mapsto \;\;
\textbf{b}_1:=  (b_1x, ...,b_Jx) \mod 1.
\end{align*}
Here $\widehat{\textbf{b}}_m$ may correspond to multiple numbers of $J$-block of $\Delta a_x(n) $,
and we pick any one of them to be $\textbf{b}_m $.
Note that $m\le L$ since $K\le LJ$.
Now it is easy to check that the map
\[\widehat{\textbf{b}} \mapsto \left( \textbf{b}_1, \textbf{b}_2,..., \textbf{b}_m, \text{NULL},...,\text{NULL} \right) \in  \left(\CBB_{J}(\Delta a_x) \cup \{\text{NULL}\}\right)^L,\]
where we add some ``NULL" if $m<L$,
is injective.
So $|\CBB_{LJ}^1(1_A)|\le (|\CBB_{J}(a_x)|+1)^L$.
Moreover,
every $LJ$-block of $1_A$ is a concatenation of at most $LJ$ 0's and one block from $\CBB_{LJ}^1(1_A)$,
so $|\CBB_{LJ}(1_A)| \le LJ|\CBB_{LJ}^1(1_A)|$,
and
\[\AE(1_A) = \lim_{J\rightarrow \infty} \frac{\log |\CBB_{LJ}(1_A)| }{LJ} \le \lim_{J\rightarrow \infty} \frac{\log ( \left(LJ|\CBB_{J}(\Delta a_x)|+1)^L \right)}{ LJ} =\AE(\Delta a_x ).
\]
Now by Theorem \ref{231211thm1} we have \eqref{240304eq1}.
\end{proof}

\section{A generalization of Furstenberg's $\times 2\times 3$ conjecture}
\label{sect:furstenberg}

In \cite{furstenberg1970},
Furstenberg proposed the following $\times 2\times 3$ conjecture:
let $p$ and $q$ be two integers $\ge 2$ with $\log p/\log q\notin \Q$,
then for every $x \in[0,1)\setminus \Q$,
$$\dim \overline{\{p^nx \bmod 1:n\in\N\}} + \dim \overline{\{q^nx \bmod 1:n\in\N \}} \geq 1.$$
In Section \ref{240115sec1}, we see
that when $a(n)=p^n$ for such an integer $p\ge 2$,
$$ \AE(a_x(n)) = \log p\cdot\dim \overline{\{p^nx \bmod 1:n\in\N\}},$$
and hence $\AE^*(a(n)) = \log p$.
So we can restate Furstenberg's conjecture as follows:
let $p$ and $q$ be two integers $\ge 2$ with $\log p/\log q\notin \Q$,
then for every $x \in[0,1) \backslash \mathbb{Q}$,
$$\frac{\AE(\{p^nx\}_{n\in\N})}{\AE^*(\{p^n\}_{n\in\N} )} + \frac{\AE(\{q^nx\}_{n\in\N})}{\AE^*(\{q^n\}_{n\in\N} )} \ge 1. $$
With the notion of dual entropy,
it is natural to extend Furstenberg's conjecture to general sequences integer $a(n)$ and $b(n)$.
\begin{conj}\label{240115prob1}
Let $\{a(n)\}_{n\in\N}$ and $\{b(n)\}_{n\in\N}$ be two centered sequences of integers, such that $\AE^*(\{a(n)\}_{n\in\N})>1$ and $\AE^*(\{b(n)\}_{n\in\N})>1$ and
$$ \frac{\AE^*(\{a(n)\}_{n\in\N})}{\AE^*(\{b(n)\}_{n\in\N})}\not\in \Q .$$
Then for every irrational $x$,
\begin{align*}
\frac{\AE(\{a(n)x\}_{n\in \N})}{ \AE^*(\{a(n)\}_{n\in\N})} + \frac{\AE(\{b(n)x\}_{n\in \N})}{ \AE^*( \{b(n)\}_{n\in\N} )} \ge 1.
\end{align*}
\end{conj}
So far, little is known for Conjecture \ref{240115prob1} except for some trivial cases.
In the following, we prove a special case of Conjecture \ref{240115prob1} under the assumption that every integer sequence is centered.

Let $p\ge 2$ be an integer, and $a(n)$ be constructed as in Example \ref{240116exm1},
that is,
$a(n) = p^n+c(n)$ for some sequence $c(n)$ such that $ \Delta c(n)$ is uniformly bounded and $\AE(\Delta c)>3\log p$.
Note that we do not require $a(n)$ to be increasing.
Then,
by Lemma \ref{231210cor1} (iii),
\begin{align*}
\AE(c_x(n)) - \AE(\{p^nx\}_{n\in \N}) \le \AE(a_x(n)) \le \AE(c_x(n)) + \AE(\{p^nx\}_{n\in \N}),\quad \forall x\in [0,1).
\end{align*}
By Lemma \ref{240117lem1},
$\AE(c_x(n)) = \AE(\Delta c)$ for every irrational $x$,
and by Proposition \ref{240115prop1},
$\AE(\{p^n x\}_{n\in \N}) \le \log p$.
So indeed, we have
\begin{align}\label{240117eq3}
\AE(\Delta c(n)) - \log p \le \AE(a_x(n)) \le \AE(\Delta c(n)) + \log p,\quad \forall x\not\in \Q.
\end{align}
Hence, if we assume $\{a(n)\}_{n\in \N}$ is centered,
then
\begin{align}\label{240117eq2}
\AE(\Delta c(n)) -\log p \le \AE^*( a(n) ) \le \AE(\Delta c(n))+\log p.
\end{align}
Recall we assume $\AE(\Delta c)>3\log p$,
then \eqref{240117eq3} and \eqref{240117eq2} implies that $$\AE(a_x(n)) > \frac12 \AE^*( a(n) )$$ for every irrational $x$.
Similarly, if we construct $b(n)$ in a similar way,
then we also have
$$\AE(b_x(n)) > \frac12 \AE^*( b(n) )$$ for every irrational $x$.
So we have the following proposition,
which partial answers Conjecture \ref{240115prob1}.

\begin{prop}\label{240117prop1}
Let $p,q\ge 2$ be two integers.
Let $a(n) = p^n+c(n)$ for some sequence $c(n)$ of integers satisfying $\Delta c(n)$ is uniformly bounded and $\AE(\Delta c)>3\log p$,
and let $b(n) = q^n+d(n)$ for some sequence $d(n)$ of integers satisfying $\Delta d(n)$ is uniformly bounded and $\AE(\Delta d)>3\log q$.
If $\{a(n)\}_{n\in \N}$ and $\{b(n)\}_{n\in \N}$ are centered,
then for every irrational $x$,
\begin{align*}
\frac{\AE(\{a(n)x\}_{n\in \N})}{ \AE^*( a(n) )} + \frac{\AE(\{b(n)x\}_{n\in \N})}{ \AE^*( b(n) )} \ge 1.
\end{align*}
\end{prop}

\section{Further discussions on the M\"{o}bius disjointness conjecture}
\label{sect:sarnak}
One motivation to study the entropy of sequences is their disjointness from certain arithmetic functions.
It is believed that all the zero entropy sequences are disjoint from certain multiplicative arithmetic functions
because zero entropy means the sequence exhibits strong regularity with respect to addition,
and the addition operator on $\N$ behaves independently to the multiplication operator in some sense.
Sarnak \cite{sarnak2011lectures} (see also \cite{liu_sarnak2015}) conjectured that all sequences arising from zero entropy dynamics systems are disjoint from the M\"{o}bius function $\mu(n)$.
It is explained in \cite{wei2022anqie} (see also \cite{ge2016}) that we can restate Sarnak's conjecture in anqie entropy,
which,
by Lemma \ref{approach by finite ranges with zero anqie entropy},  is equivalent to the following statement:
for every subset $A\subseteq \N$ with $\AE(1_A(n))=0$ we have
\[ \lim_{N\rightarrow \infty } \frac 1N \sum_{n=1}^N 1_A (n) \mu(n) =0 .\]
Since the notion,
dual entropy,
also describes the complexity of subsets $A$ of $\N$,
we believe that the sequences with zero dual entropy are disjoint from the M\"{o}bius function.
\begin{conj}\label{240117conj2}
For every subset $A\subseteq \N$ with $\AE^*(A)=0$, we have
\[ \lim_{N\rightarrow \infty } \frac 1N \sum_{n=1}^N 1_A (n) \mu(n) =0 .\]
\end{conj}
Many sequences satisfying Sarnak's conjecture also have dual entropy zero.
Simple examples include periodic sequences and characteristic functions supported on zero-density sets.
Another class of $\{0,1\}$-sequences known to satisfy
Sarnak's conjecture is constructed from polynomials,
or generated polynomials as defined in \cite{bergelson2007distribution}  (see \cite{green_tao2012mobius}.)
For example,
letting $p(n)$ be a polynomial or a generated polynomial,
define
\[a(n) = 1 \text{ if } p(n)\bmod 1\in [0,1/2);\quad  a(n) = 0 \text{ if } p(n)\bmod 1\in [1/2,1).\]
Then it is known that the sequence $a(n)$ has zero (anqie) entropy and is disjoint from $\mu(n)$.

The M\"{o}bius disjointness conjecture suggests that the entropy of a sequence is related to its disjointness from $\mu(n)$,
and entropic van der Corput's Theorem 
says that the entropy of a sequence will not change if we take the difference.
Hence it is natural to ask that whether the difference operator keeps the disjointness of a sequence from $\mu(n)$.
The following proposition answers this question negatively
by constructing a sequence $a(n)$ which is not disjoint from $\mu(n)$ and whose difference does.

\begin{prop}
There is a sequence $a(n)$ such that
\[\lim_{N\rightarrow \infty } \frac 1N \sum_{n=1}^N \Delta a(n) \mu(n)=0,\]
but
\[ \liminf_{N\rightarrow \infty } \frac 1N \sum_{n=1}^N a (n) \mu(n) \ge \frac{2}{\pi^2}>0 .\]
\end{prop}
\begin{proof}

We will construct a sequence $a(n)$ such that,
\begin{align}\label{240113eq2}
\sum_{j=1}^4\Delta a(4k+j)\mu(4k+j)=0\quad \text{ and } \sum_{j=1}^4a(4k+j)\mu(4k+j)\ge 1,
\end{align}
unless
$(\mu(4k), ...,\mu(4k+3)) = (0,0,0,0).$

We construct $a(n)$ by determining all the regular 4-blocks of $a(n)$ and $\Delta a(n)$.
Note that $\mu(4k)=0$ for every $k$.
When
\[\sum_{j=1}^4 \mu(4k+j)\neq 0,\]
we pick either
\begin{align*}
(a(4k), ...,a(4k+3)) = (*,1,1,1), \quad
(\Delta a(4k), ..., \Delta a(4k+3)) = (*,0,0,0)
\end{align*}
or
\[(a(4k), ...,a(4k+3)) = (*,-1,-1,-1),\quad (\Delta a(4k),..., \Delta a(4k+3)) = (*,0,0,0) \]
to make sure that
\[\sum_{j=1}^4a(4k+j)\mu(4k+j)\ge 1.\]
Here and in the following, the notation ``$*$" in $a(n)$ means the value is determined by $a(n-1)$ and $\Delta a(n-1)$,
and $*$ in $\Delta a(n)$ means the value is determined by $a(n)$ and $ a(n-1)$.
When
\[\sum_{j=1}^4 \mu(4k+j)= 0,\]
there are only 6 cases for $(\mu(4k), ...,\mu(4k+3))$ if they are not all 0.
We list these cases in the following table,
and for each case we give a construction of $(a(4k), ...,a(4k+3))$ and $
(\Delta a(4k), ..., \Delta a(4k+3))$.

{
\renewcommand{\arraystretch}{1.25}
\begin{table}[H]
    \centering
    \begin{tabular}{c|c|c}
\toprule
$\left(\mu(4n), \dots, \mu(4n+3)\right)$
&$\left(a(4n), \dots, a(4n+3)\right)$
&$\left(\Delta a(4n), \dots, \Delta a(4n+3)\right)$\\
\hline
$\begin{array}{@{(}w{r}{0.8em}@{, }w{r}{1.5em}@{, }w{r}{1.5em}@{, }w{r}{1.5em}@{)}}
0   &-1 &0  &1\\
0   &-1 &1  &0\\
0   &0  &-1 &1\\
0   &0  &1  &-1\\
0   &1  &-1 &0\\
0   &1  &0  &-1
\end{array}$ &
$\begin{array}{@{(}w{r}{0.8em}@{, }w{r}{1.5em}@{, }w{r}{1.5em}@{, }w{r}{1.5em}@{)}}
*   &-1 &-1 &0\\
*   &-1 &0  &1\\
*   &0  &-1 &0\\
*   &0  &1  &0\\
*   &1  &0  &-1\\
*   &1  &0  &0
\end{array}$ &
$\begin{array}{@{(}w{r}{0.8em}@{, }w{r}{1.5em}@{, }w{r}{1.5em}@{, }w{r}{1.5em}@{)}}
*   &0  &1  &0\\
*   &1  &1  &0\\
*   &-1 &1  &1\\
*   &1  &-1 &1\\
*   &-1 &-1 &0\\
*   &-1 &0  &-1
\end{array}$ \\
\bottomrule
    \end{tabular}
\end{table}
}

Finally, let $a(0)=0$.
It is easy to check that the above values of the regular 4-blocks of $a(n)$ and $\Delta a(n)$ determine a unique bounded sequence $a(n)$,
and the condition \eqref{240113eq2} is satisfied.
Hence
\[\lim_{N\rightarrow \infty } \frac 1N \sum_{n=1}^N \Delta a(n) \mu(n)=0,\]
and
\[\liminf_{N\rightarrow \infty } \frac 1N \sum_{n=1}^N a (n) \mu(n) \ge \lim_{N\rightarrow \infty } \frac 1N \sum_{n=1}^N \frac{|\mu(n)|}3 = \frac{2}{\pi^2} .\]
Hence $a(n)$ is a needed sequence.
\end{proof}

Hence we cannot use (anqie) entropy to completely determine whether a sequence is disjoint from $\mu(n)$ or not.

\section*{Acknowledgments} The first author would like to thank Xianjin Li for many inspiring discussions
during a visit to him.
The authors would also like to thank Liming Ge for helpful conversations.

\appendix

\section{Proof of Lemma \ref{231210lem1}}\label{appendixA}
Given $N \geq 1$.
Let
\[U_i = \left( \frac i N - \frac{1+\epsilon}{2N}, \; \frac i N + \frac{1+\epsilon}{2N}\right), \]
then $\left\{U_i : i=0,1, \ldots, N-1 \right\}$ is an open cover of $\R/\Z$.
We denote $X$ to be the closure of the set $\{(x(n), x(n+1), \ldots): n \in \mathbb{N}\}$ in $()^{\mathbb{N}}$.
Let $B_x$ be the Bernoulli shift on $X$ given by $\left(\omega_0, \omega_1, \ldots\right) \mapsto\left(\omega_1, \omega_2, \ldots\right)$. For $s \geq 1$, denote
$$
\mathcal{W}_s := \left\{U_{i_0} \times U_{i_1} \times \ldots \times U_{i_{s-1}} \times ()^{\mathbb{N} \backslash\{0,1, \ldots, s-1\}}: i_0, i_1, \ldots, i_{s-1}  \in\{0,1, \ldots, N\} \right\},
$$
so $\mathcal{W}_s$ is an open cover of $X$.

For $l=0,1,...$, we iteratively construct integers $t_l$, $s_l$ and open covers $\mathcal{U}^{(l)}$.
We start from letting $t_0=1$ and $\mathcal{U}^{(0)} = \mathcal{W}_1$.
At the $l$-th step,
assume we have a natural number $t_l$ and an open cover $\mathcal{U}^{(l)}$ of $X$ which is a subcover of $\mathcal{W}_{t_l}$.
From the definition of the anqie entropy of $x(n)$,
the topological entropy of $B_x$,
denoted by $h\left(B_x\right)$, is equal to   $\lambda$.
So $h\left(B_x^{t_l}\right)=t_l \lambda$.
Therefore,
$$\lim _{s \rightarrow \infty} s^{-1} \log \mathcal{N}\left( \bigvee_{j=0}^{s-1} \left(B_x^{t_l}\right)^{-j} \mathcal{U}^{(l)} \right) \leq t_l \lambda,$$
where recall we use $\mathcal{N}(\mathcal{U})$ to denote the minimal number of open sets we need in the cover $\mathcal{U}$ to cover $X$.
So there is a sufficiently large natural number
$s_l$ such that
\begin{align}\label{240111eq1}
s_l^{-1} \log \mathcal{N}\left( \bigvee_{j=0}^{s_l-1} \left(B_x^{t_l}\right)^{-j} \mathcal{U}^{(l)} \right) \leq t_l \lambda +2^{-l}.
\end{align}
Therefore it is possible to choose a subcover
\begin{align}\label{240111eq4}
\mathcal{U}^{(l+1)} \subseteq \bigvee_{j=0}^{s_l-1} \left(B_x^{t_l}\right)^{-j} \mathcal{U}^{(l)}
\end{align}
such that its size $\left|\mathcal{U}^{(l+1)}\right| $ satisfies
\begin{align}\label{240111eq2}
s_l^{-1} \log \left|\mathcal{U}^{(l+1)}\right| \leq t_l \lambda +2^{-l}.
\end{align}
We set $t_{l+1}=t_l s_l$.
Then clearly $\mathcal{U}^{(l+1)}$  is a subcover of $\mathcal{W}_{t_{l+1}}$.

To make it clear,
we also construct a sequence of functions $$f_l: \N\rightarrow \{0,\frac{1}{N},\frac{2}{N},..., \frac{N-1}{N}\},\quad l=0,1,2,...$$
We define the values of $f_l$ by its regular $t_{l}$-blocks.
For the $j$-th regular $t_{l}$-block,
since $\mathcal{U}^{(l)}$ is an open cover of $X$,
we can choose an open set $V_j^{(l)}$ from $\mathcal{U}^{(l)}$ such that the point
\begin{align}\label{240111eq3}
\left(x\left(j t_{l}\right), x\left(j t_{l}+1\right), \ldots \right) \in V_j^{(l)}.
\end{align}
Since $\mathcal{U}^{(l)}$  is a subcover of $\mathcal{W}_{t_{l}}$,
we can write $V_j^{(l)}$ as the form
\[V_j^{(l)}= U_{i_0} \times U_{i_1} \times \cdots \times U_{i_{t_{l}-1}} \times X_0^{\mathbb{N} \backslash\left\{0,1, \ldots, t_{l}-1\right\}} . \]
Then we define
\[
\left(f_l\left(j t_{l}\right), f_l\left(j t_{l}+1\right),\ldots, f_l\left((j+1) t_{l}-1\right)\right) = \left(\frac {i_0} N,  \frac {i_1} N , \ldots, \frac {i_{t_{l}-1}} N \right).
\]
Let $j$ run over $\N$ then we complete the construction of $f_l$.
Note that by \eqref{240111eq3},
for any $n,l$ we have
\begin{align}\label{240111eq5}
 |x(n) - f_l(n)| < \frac{1+\epsilon}{2N}.
\end{align}

Finally, we construct the map $g_N$.
Define $g_N(0) = f_0(0) $,
and
\[g_N(n)=f_l(n) \quad \text{ when } \quad t_l \leq n<t_{l+1} , l\ge 0.\]
From \eqref{240111eq5} we have
$\sup_{n\in \mathbb{N}}|g_{N}(n)-x(n)|\le  \frac{1+\epsilon}{2N}$.
We estimate the entropy of $g_N$ by considering its effective regular $t_l$-blocks.
By \eqref{240111eq4} and the construction of $f_l$ and $g_N$,
every effective regular $t_l$-block of $g_N$
is defined from an element of $\mathcal{U}^{(l)}$,
hence $g_N(n)$ has at most $\left|\mathcal{U}^{(l+1)}\right|$ different effective regular $t_l$-blocks.
By Lemma \ref{lem:Jblock},
$$
\AE(g_N)=\lim_{l \rightarrow \infty} \frac{1}{t_{l}} \log   \left|B_{t_l}^{e,r}(g_N)\right| \leq \lim_{l \rightarrow \infty} \frac{1}{t_{l}} \log   \left|\mathcal{U}^{(l+1)}\right| \le \lambda,
$$
where the last inequality comes from \eqref{240111eq2} and the definition $t_{l+1}=t_l s_l$.
Hence the proof is completed.

\bibliographystyle{elsarticle-num}
\bibliography{ref.bib}

\end{document}